\documentclass[a4paper,11pt]{amsart}
\usepackage{amsmath,amssymb,amsthm}
\usepackage{amscd}
\usepackage{array,enumerate}
\newtheorem{Thm}{Theorem}[section]
\newtheorem{Prop}[Thm]{Proposition}
\newtheorem{Lem}[Thm]{Lemma}
\newtheorem{Cor}[Thm]{Corollary}

\newtheorem{MThm}{Main Theorem}

\theoremstyle{definition}
\newtheorem{Rem}[Thm]{Remark}

\newcommand{\Cs}{\mbox{${\rm C}^\ast$}}
\newcommand{\id}{\mbox{\rm id}}

\title[Minimal ambient nuclear \Cs -algebras]{Minimal ambient nuclear \Cs -algebras}
\author{Yuhei Suzuki}
\address{Department of Mathematical Sciences,
University of Tokyo, Komaba, Tokyo, 153-8914, Japan}
\email{suzukiyu@ms.u-tokyo.ac.jp}
\begin{document}
\begin{abstract}
We provide examples of ambient nuclear \Cs -algebras of non-nuclear \Cs -algebras
with no proper intermediate \Cs -algebras.
In particular this gives the first examples of minimal ambient nuclear \Cs -algebras of
non-nuclear \Cs -algebras.
For this purpose, we study generic Cantor systems of infinite free product groups.
\end{abstract} 
\maketitle
\section{Introduction}\label{Sec:intro}
In 1979, Choi \cite{Cho} constructed the first example of an ambient nuclear \Cs -algebra of a non-nuclear \Cs -algebra.
In the celebrated paper of Kirchberg--Phillips \cite{KP},
they show that any separable exact \Cs -algebra in fact has an ambient nuclear \Cs -algebra.
(In fact, one can choose it to be isomorphic to the Cuntz algebra $\mathcal{O}_2$.)
When we consider reduced group \Cs -algebras,
thanks to Ozawa's result \cite{Oz}, we have more natural ambient nuclear \Cs -algebras,
namely, the reduced crossed products of amenable dynamical systems.
Ambient nuclear \Cs -algebras play important roles in theory of both \Cs- and von Neumann algebras.
We refer the reader to the books \cite{BO} and \cite{Ror} for details.
In this paper, we investigate how an ambient nuclear \Cs -algebra of a non-nuclear \Cs -algebra
can be tight. See the paper \cite{Oz1} for a related topic.
Based on (new) results on topological dynamical systems,
we give the first example of a minimal ambient nuclear \Cs-algebra of a non-nuclear \Cs -algebra. 
In fact, we have a stronger result: our examples of minimal ambient nuclear \Cs-algebras have no proper intermediate \Cs -algebras.

Note that as shown in \cite{Suz2},
in contrast to injectivity of von Neumann algebras,
nuclearity of \Cs -algebras is not preserved under taking the decreasing intersection.
We also note that the increasing union of non-nuclear \Cs -algebras
can be nuclear. See Remark \ref{Rem:ma} for the detail.
Thus there is no obvious way to provide a minimal ambient nuclear \Cs-algebra.
We also remark that in the von Neumann algebra case, thanks to the bicommutant theorem, for any von Neumann algebra,
finding a minimal ambient injective von Neumann algebra is equivalent
to finding a maximal injective von Neumann subalgebra.
The latter problem draws many reserchers's interest.
In \cite{Pop}, Popa provided the first concrete examples of maximal injective von Neumann subalgebras.
For recent progresses on this problem, see \cite{BC} and \cite{Hou} for instance.

In 1975, Powers \cite{Pow} invented a celebrated method to study structures of the reduced group
\Cs -algebras. His idea has been applied to more general situations, particularly for reduced crossed products,
and to more general groups, by many hands.
See \cite{dlS} for instance.
We combine his technique with certain properties of dynamical systems to obtain the following main theorem of the paper.

We say that a group is an infinite free product group
if it is a free product of infinitely many nontrivial groups.
Throughout the paper, groups are supposed to be countable.
\begin{MThm}[Corollary \ref{Cor:gen}, Theorem \ref{Thm:main}]
Let $\Gamma$ be an infinite free product group with the AP $($approximation property$)$ \cite{HK}
$($or equivalently, each free product component has the AP$)$.
Then there is an amenable action of $\Gamma$ on the Cantor set $X$
with the following property.
There is no proper intermediate \Cs-algebra of the inclusion
${\rm C}^\ast_ r(\Gamma)\subset C(X)\rtimes_r \Gamma$.
In particular $C(X)\rtimes_r \Gamma$ is a minimal
ambient nuclear \Cs -algebra of the non-nuclear \Cs -algebra ${\rm C}^\ast_r(\Gamma)$.
\end{MThm}
We remark that it is not known if there is an ambient injective von Neumann algebra (or equivalently, injective von Neumann subalgebra) of a non-injective von Neumann algebra with no proper intermediate von Neumann algebra.
We refer the reader to \cite{BO} for basic knowledge on \Cs -algebras and discrete groups.
Here we remark that the AP implies exactness (see Section 12.4 of \cite{BO}),
while the converse is not true \cite{LS}.
In Main Theorem, we need the AP to determine when a given element of the reduced crossed product sits in the reduced group \Cs -algebra. (Cf.~ \cite{Suz2}, \cite{Zac}.)

In theory of both measurable and topological dynamical systems,
the Baire category theorem is a powerful tool
to produce an example with a nice property.
For further information on this topic, see the paper \cite{Hoc} and the introduction therein for instance.
We follow this strategy to construct dynamical systems as in Main Theorem.
To apply the Baire category theorem, we need
a nice topology on the set of dynamical systems.
In this paper, we deal with the following (well-known) space of topological dynamical systems.
Let $X$ be a compact metric space with a metric $d_X$.
Then, on the homeomorphism group ${\rm Homeo}(X)$ of $X$,
define a metric $d$ as follows.
\[d(\varphi, \psi):=\max_{x\in X}d_X(\varphi(x), \psi(x))+\max_{x\in X} d_X(\varphi^{-1}(x), \psi^{-1}(x)).\]
Then $d$ defines a complete metric on ${\rm Homeo}(X)$.
The topology defined by $d$
coincides with the uniform convergence topology.
In particular it does not depend on the choice of $d_X$.
Next let $\Gamma$ be a (countable) group and
consider the set $\mathcal{S}(\Gamma, X)={\rm Hom}(\Gamma, {\rm Homeo}(X))$ of all dynamical systems of $\Gamma$ on $X$.
The set $\mathcal{S}(\Gamma, X)$ is naturally identified with a closed subset of the product space
$\prod_\Gamma {\rm Homeo}(X)$, where the latter space is equipped with the product topology.
Since $\Gamma$ is countable, this makes $\mathcal{S}(\Gamma, X)$
a complete metric space.

Finally, we recall some definitions from theory of topological dynamical systems.
Let $\alpha\colon \Gamma \curvearrowright X$ and $\beta \colon \Gamma \curvearrowright Y$
be actions of a group on compact metrizable spaces.
The $\alpha$ is said to be an extension of $\beta$
if there is a $\Gamma$-equivariant quotient map $\pi\colon X\rightarrow Y$.
In this case $\beta$ is said to be a factor of $\alpha$.
The action $\alpha\colon \Gamma \curvearrowright X$ is said to be
\begin{itemize}
\item free if
any $s\in \Gamma \setminus\{e\}$ has no fixed points,
\item minimal
if every $\Gamma$-orbit is dense in $X$,

\item prime if there is no nontrivial factor of $\alpha$,
\item amenable if
for any $\epsilon>0$ and any finite subset $S$ of $\Gamma$,
there is a continuous map
$\mu\colon X\rightarrow {\rm Prob}(\Gamma)$ satisfying
$\|s.\mu^x-\mu^{s.x}\|_1<\epsilon$ for all $s\in S$ and $x\in X.$
\end{itemize}
Here ${\rm Prob}(\Gamma)$ denotes the space of probability measures on $\Gamma$
equipped with the pointwise convergence topology (which coincides with the $\ell^1$-norm topology),
and $\Gamma$ acts on ${\rm Prob}(\Gamma)$ by the left translation.
Obviously freeness and amenability pass to extensions and minimality passes to factors.
Anantharaman-Delaroche \cite{Ana} has characterized amenability of topological dynamical systems by the nuclearity of the reduced crossed product.

We say that a property of topological dynamical systems is open, $G_\delta$, dense, $G_\delta$-dense, respectively
when the subset of $\mathcal{S}(\Gamma, X)$ consisting of actions with this property
has the corresponding property.
We say that a property is generic when the corresponding set contains a $G_\delta$-dense subset of $\mathcal{S}(\Gamma, X)$.
Note that thanks to the Baire category theorem,
the intersection of countably many $G_\delta$-dense properties is again $G_\delta$-dense,
and similarly for genericity.
Although some results (e.g., genericity of amenability, minimality, primeness, for infinite free product groups)
can be extended to more general spaces by minor modifications,
we concentrate on the Cantor set.
This is enough for Main Theorem.
For short, we call an action on the Cantor set a Cantor system.

\subsection*{Notation}
Most notation we use are standard ones.
\begin{itemize}
\item For an action $\alpha\colon \Gamma \curvearrowright X$,
let $C(X)\rtimes_{\rm alg}\Gamma$ denote its algebraic crossed product,
i.e., the $\ast$-subalgebra of the reduced crossed product generated by
$C(X)$ and $\Gamma$.
\item
For the simplicity of notation,
in the reduced crossed product $A=C(X)\rtimes_r \Gamma$,
we denote the unitary of $A$ corresponding to $s\in \Gamma$
by the same symbol $s$.
\item
Denote by $e$ the unit element of a group.
\item
Let $E\colon C(X)\rtimes_r \Gamma \rightarrow C(X)$ denote the canonical conditional expectation on the reduced crossed product.
That is, the unital completely positive map
defined by the formula $E(fs):=\delta_{e, s}f$ for $f\in C(X)$ and $s\in \Gamma$.
\item For a unital \Cs -algebra, we denote by
$\mathbb{C}$ the \Cs -subalgebra generated by the unit.
\item Denote by $\otimes$ the minimal tensor product of \Cs -algebras.
We use the same notation for the minimal tensor product of completely positive maps.
\item For a subset $S$ of a set,
denote by $\chi_S$ the characteristic function of $S$.
\item
For a subset $S$ of a group,
denote by $\langle S \rangle$ the subgroup generated by $S$.
\item
When the action $\alpha\colon \Gamma\curvearrowright X$ is clear from the context,
we denote $\alpha_s(x)$ by $s.x$ for short.
Similarly for $s\in \Gamma$ and $U \subset X$,
we denote $\alpha_s(U)$ by $sU$ when no confusion arises.

\end{itemize}
\section{Some generic properties of Cantor systems}\label{sec:gen}
In this section, we summarize generic properties of Cantor systems.
From now on we denote by $X$ the Cantor set.
We recall that the Cantor set is the topological space
characterized (up to homeomorphism) by the following four properties:
compactness, total disconnectedness, metrizability, and perfectness (i.e., no isolated points).
\begin{Lem}\label{Lem:Gdelta}
For any group $\Gamma$, the following properties are $G_\delta$ in $\mathcal{S}(\Gamma, X)$.
\begin{enumerate}[\upshape (1)]
\item Freeness.
\item Amenability.
\end{enumerate}
\end{Lem}
\begin{proof}
The first claim is well-known.
For completeness, we include a proof.

\noindent
(1): For $s\in \Gamma$,
set $V_s:=\{\alpha \in \mathcal{S}(\Gamma, X): \alpha_s(x) \neq x {\rm\ for\ all\ }x\in X\}$.
By the compactness of $X$, each $V_s$ is open.
The $G_\delta$-set
$\bigcap_{s\in \Gamma \setminus \{e\}} V_s$ consists of all free Cantor systems.

\noindent
(2): For each finite subset $S$ of $\Gamma$, we say that an action $\alpha \colon \Gamma \curvearrowright X$ has property $\mathcal{A}_S$ if
it admits a continuous map
$\mu\colon X \rightarrow {\rm Prob}(\Gamma)$
satisfying $$\|s.\mu^x - \mu^{s.x}\|_1<\frac{1}{|S|}$$
for all $s\in S$ and $x\in X$.
Let $\alpha\in \mathcal{S}(\Gamma, X)$ be given and suppose we have a continuous map $\mu$ that witnesses $\mathcal{A}_S$ of $\alpha$.
Then, by the continuity of $\mu$, it guarantees $\mathcal{A}_S$ for any $\beta$ sufficiently close to $\alpha$.
This shows that $\mathcal{A}_S$ is open.
Now obviously, the intersection $\bigwedge_S \mathcal{A}_S$ is equivalent to amenability,
where $S$ runs over finite subsets of $\Gamma$.
\end{proof}

The following simple lemma is crucial to show the genericity of some properties.
\begin{Lem}\label{Lem:prod}
Let $\alpha \colon \Gamma \curvearrowright X$ be a given Cantor system.
Then the set of extensions of $\alpha$
is dense in $\mathcal{S}(\Gamma, X)$.
\end{Lem}
\begin{proof}
Let us regard the Cantor set $X$ as the direct product of infinitely many copies $Y$ of the Cantor set:
$X=Y{^\mathbb{N}}$.
We regard $\alpha$ as a dynamical system on $Y$ via a homeomorphism $X\cong Y$.
For each $N \in \mathbb{N}$,
define a map $\sigma_N\colon \mathbb{N} \rightarrow \mathbb{N}$
by
\begin{eqnarray*}
 \sigma_N(n):=\left\{ \begin{array}{ll}
 n & {\rm\ when \ } n< N,\\
 n+1 & {\rm\ when \ }n\geq N.\\
 \end{array} \right.
 \end{eqnarray*} 
Now let $\beta\in \mathcal{S}(\Gamma, X)$ be given.
Let $\gamma\colon \Gamma \curvearrowright Y \times X$ be the diagonal action of $\alpha$ and $\beta$.
For each $N\in \mathbb{N}$, define a homeomorphism $\varphi_N \colon X \rightarrow Y \times X$
by $\varphi_N(x):=(x_N, (x_{\sigma_N(n)})_{n \in \mathbb{N}})$.
Put $\beta^{(N)}:=\varphi_N^{-1} \circ \gamma \circ\varphi_N \in \mathcal{S}(\Gamma, X)$.
Then for each $N\in \mathbb{N}$, the projection from $X$ onto the $N$th coordinate gives
a factor map of $\beta^{(N)}$ onto $\alpha$.
Moreover the sequence $(\beta^{(n)})_{n=1}^\infty$ converges to $\beta$.
Since $\beta$ is arbitrary, this proves the claim.
\end{proof}
The next lemma is well-known.
For completeness, we give a proof.
\begin{Lem}\label{Lem:can}
Every group admits a free Cantor system.
Also, every exact group admits an amenable Cantor system.
\end{Lem}
\begin{proof}
Let $\Gamma$ be a group.
We first show that the left translation action
of $\Gamma$ on its Stone--\v{C}ech compactification $\beta \Gamma$ is free.
Let $s\in \Gamma \setminus\{ e\}$ be given. Put $\Lambda:=\langle s \rangle$.
Take a $\Lambda$-equivariant map
$\Gamma \rightarrow \Lambda$
where $\Lambda$ acts on both groups by the left multiplication.
This extends to the $\Lambda$-equivariant quotient map
$\beta\Gamma \rightarrow \beta \Lambda$.
By universality,
$\beta\Lambda$ factors onto every minimal dynamical system of $\Lambda$ (on a compact space).
Since any cyclic group admits a minimal free action on a compact space,
this shows that $s$ has no fixed points in $\beta \Gamma$.

Let $(A_\mu)_{\mu \in M}$ be the increasing net of
$\Gamma$-invariant unital \Cs -subalgebras of $\ell^\infty(\Gamma)=C(\beta \Gamma)$ generated by countably many projections.
Note that $\bigcup_{\mu \in M} A_\mu = \ell^\infty (\Gamma)$.
Let $X_\mu$ denote the spectrum of $A_\mu$.
Obviously, each $X_\mu$ is totally disconnected and metrizable.
Let $\alpha_\mu \colon \Gamma \curvearrowright X_\mu$ be the action induced from
the action $\Gamma \curvearrowright A_\mu$.
By the freeness of $\Gamma \curvearrowright \beta \Gamma$, for sufficiently large $\mu$,
the $\alpha_\mu$ must be free.
When $\Gamma$ is exact,
then as stated in Theorem 5.1.7 of \cite{BO}, for sufficiently large $\mu$,
the $\alpha_\mu$ must be amenable.
Hence for sufficiently large $\mu$,
the diagonal action of $\alpha_\mu$ and the trivial Cantor system gives the desired action.
\end{proof}
We now summarize the results of this section.
\begin{Cor}\label{Cor:gen}
For any group $\Gamma$, freeness is a $G_\delta$-dense property in $\mathcal{S}(\Gamma, X)$.
Moreover, when $\Gamma$ is exact,
amenability is also a $G_\delta$-dense property in $\mathcal{S}(\Gamma, X)$.
\end{Cor}
\begin{proof}
Since both freeness and amenability are inherited to extensions,
it follows from Lemmas \ref{Lem:Gdelta} through \ref{Lem:can}.
\end{proof}

\section{Construction of dynamical systems and proof of Main Theorem}\label{sec:main}
In this section, we prove Main Theorem.
Let $(\Gamma_i)_{i=1}^\infty$ be a sequence of nontrivial groups and
let $\Gamma:=\lower0.25ex\hbox{\LARGE $\ast$}_{i=1}^\infty \Gamma_i$ be their free product.
By replacing $\Gamma_i$ by $\Gamma_{2i-1}\ast \Gamma_{2i}$ for all $i$ if necessary,
in the rest of the paper, we assume that each free product component $\Gamma_i$
contains a torsion-free element.
We start with the following elementary lemmas.
We remark that in the case that $\Gamma$ is the free group $\mathbb{F}_\infty$,
we do not need these lemmas.
\begin{Lem}\label{Lem:ext}
Let $\Lambda$ be a group and $\Upsilon$ be its subgroup.
Then for any minimal dynamical system $\alpha$ of $\Upsilon$ on a compact metrizable space,
there is a Cantor system of $\Lambda$
whose restriction on $\Upsilon$ is an extension of $\alpha$.
\end{Lem}
\begin{proof}
Let $\alpha \colon \Upsilon \curvearrowright Y$ be an action as in the statement.
Fix an element $y\in Y$.
Then the map $\Upsilon \rightarrow Y$
defined by $s\mapsto s.y$ extends to
a factor map $\beta \Upsilon \rightarrow Y$.
This induces an $\Upsilon$-equivariant unital embedding of $C(Y)$ into $\ell^\infty(\Upsilon)$.
By the right coset decomposition of $\Lambda$ with respect to $\Upsilon$,
we have an $\Upsilon$-equivariant unital embedding of $\ell^\infty(\Upsilon)$ into $\ell^\infty(\Lambda)$.
We identify $C(Y)$ with a unital $\Upsilon$-invariant \Cs -subalgebra
of $\ell^\infty(\Lambda)$ via the composite of these two embeddings.
Take a $\Lambda$-invariant \Cs-subalgebra $A$ of $\ell^\infty(\Lambda)$ which contains $C(Y)$
and is generated by countably many projections.
Let $Z$ be the spectrum of $A$. Note that $Z$ is metrizable and totally disconnected.
Let $\beta \colon \Lambda \curvearrowright Z$ be the action
induced from the action $\Lambda \curvearrowright A$.
Since $A$ contains $C(Y)$ as a unital \Cs -subalgebra,
the restriction of $\beta$ on $\Upsilon$ is an extension of $\alpha$.
Now the diagonal action of $\beta$ with the trivial Cantor system
gives the desired Cantor system.
\end{proof}
\begin{Lem}\label{Lem:R}
Let $\Lambda$ be a group. Let $s$ be a torsion-free element of $\Lambda$.
Then for any finite family $\mathcal{U}=\{U_1, \ldots, U_n \}$ of pairwise disjoint proper clopen subsets of $X$,
there is a Cantor system $\alpha \colon \Lambda \curvearrowright X$ with
$s U_i= U_{i+1}$ for all $i$.
Here and below, we put $U_{n+1}:=U_1$ for convenience.
\end{Lem}
\begin{proof}
By Lemma \ref{Lem:ext}, there is a Cantor system $\alpha\colon \Lambda \curvearrowright X$
whose restriction on $\langle s \rangle$ factors a transitive action on the set $\{1, \ldots, n\}$.
For such $\alpha$, there is a partition $\{V_1, \ldots, V_n\}$ of $X$ by clopen subsets satisfying
$sV_i=V_{i+1}$ for all $i$.
Set $I:= \{0, 1\}$ if $\bigcup_{i=1}^n U_i \neq X$.
Otherwise we set $I:= \{0 \}$.
Then define a new action $\beta \colon \Lambda \curvearrowright X \times I$
by \begin{eqnarray*}
 \beta_t(x, j):=\left\{ \begin{array}{ll}
 (\alpha_t(x), 0) & {\rm\ when \ } j=0,\\
 (x, 1) &  {\rm\ otherwise}.\\
 \end{array} \right.
 \end{eqnarray*}
Since nonempty clopen subsets of the Cantor set are mutually homeomorphic,
there is a homeomorphism $\varphi \colon X\times I \rightarrow X$
which maps $V_i\times \{ 0 \}$ onto $U_i$ for each $i$.
For such $\varphi$, the conjugate $\varphi\circ \beta \circ \varphi^{-1}$ gives the desired Cantor system.
\end{proof}
We next introduce a property of Cantor systems which is one of the key of the proof of Main Theorem
and show that this property is $G_\delta$-dense for infinite free product groups.
\begin{Prop}\label{Prop:P}
Let $\Gamma= \lower0.25ex\hbox{\LARGE $\ast$}_{i=1}^\infty \Gamma_i$ be an infinite free product group.
Then the following property $\mathcal{R}$ of Cantor systems is $G_\delta$-dense in $\mathcal{S}(\Gamma, X)$.
\begin{itemize}
\item[{\rm (}$\mathcal{R}${\rm ):}]For any finite family $\mathcal{U}=\{U_1, \ldots, U_n\}$ of
mutually disjoint proper clopen subsets of $X$,
there are infinitely many $i\in \mathbb{N}$ satisfying the following condition.
The group $\Gamma_i$ contains a torsion-free element $s$ satisfying
$sU_j=U_{j+1}$ for all $j$.
\end{itemize}
Here we put $U_{n+1}:=U_1$ as before.
\end{Prop}
\begin{proof}
For any $i \in \mathbb{N}$ and a family $\mathcal{U}$ as stated,
we say that an element $\alpha\in \mathcal{S}(\Gamma, X)$ has property
$\mathcal{R}(i, \mathcal{U})$ if it satisfies the following condition.
There are $k \geq i$ and a torsion-free element $s\in \Gamma_k$ satisfying
$sU_j=U_{j+1}$ for all $j$.
Then observe that for any two clopen subsets $U$ and $V$ of $X$,
the set
\[\{\varphi\in {\rm Homeo}(X): \varphi(U)=V\}\]
is clopen in ${\rm Homeo}(X)$.
This shows that property $\mathcal{R}(i, \mathcal{U})$
is open in $\mathcal{S}(\Gamma, X)$.

To show the density of $\mathcal{R}(i, \mathcal{U})$,
for each $m \in \mathbb{N}$, take a Cantor system
$\varphi_m \colon \Gamma_m \curvearrowright X$ as in Lemma \ref{Lem:R}.
Let $\alpha \in \mathcal{S}(\Gamma, X)$ be given.
Then, for each $m\in \mathbb{N}$,
we define $\alpha^{(m)} \in \mathcal{S}(\Gamma, X)$ as follows.
\begin{eqnarray*}
 \alpha^{(m)}|_{\Gamma_k}:=\left\{ \begin{array}{ll}
 \alpha|_{\Gamma_k}& {\rm\ for \ }k< m,\\
 \varphi_k & {\rm\ for \ }k\geq m.\\
 \end{array} \right.
 \end{eqnarray*} 
Then each $\alpha^{(m)}$ satisfies property $\mathcal{R}(i, \mathcal{U})$
and the sequence $(\alpha^{(m)})_{m=1}^\infty$ converges to $\alpha$.
This proves the density of $\mathcal{R}(i, \mathcal{U})$.

Now observe that property $\mathcal{R}$ is equivalent to
the intersection $\bigwedge_{i, \mathcal{U}} \mathcal{R}(i, \mathcal{U})$.
Since there are only countably many clopen subsets in $X$,
the intersection is taken over a countable family.
Now the Baire category theorem completes the proof.
\end{proof}
\begin{Rem}\label{Rem:P}
It is not hard to check that
an action with $\mathcal{R}$ is a boundary in Furstenberg's sense
(see Definition 3.8 of \cite{KK} for the definition).
Also, by Theorem 5 of \cite{LaS}, property $\mathcal{R}$ with topological freeness implies
the pure infiniteness of the reduced crossed product.
\end{Rem}
\begin{Rem}
Since every infinite group admits a weak mixing Cantor system of all orders (e.g., the Bernoulli shift),
in a similar way to the proof of Proposition \ref{Prop:P}, it can be shown that weak mixing of all orders is
$G_\delta$-dense for infinite free product groups.
Here recall that a topological dynamical system $\alpha$ is said to be weak mixing of all orders
if for any $n\in \mathbb{N}$, the diagonal action of $n$ copies of $\alpha$ has a dense orbit.
Similarly, it can also be shown that the set of disjoint pairs $(\alpha, \beta) \in \mathcal{S}(\Gamma, X)^2$ is generic in $\mathcal{S}(\Gamma, X)^2$.
Here recall that two minimal dynamical systems are disjoint
if and only if their diagonal action is minimal.
\end{Rem}
\begin{Rem}\label{Rem:K}
Consider the case $\Gamma =\mathbb{F}_\infty$.
Then by the Pimsner--Voiculescu exact sequence \cite{PV},
property $\mathcal{R}$ implies $K_0(C(X)\rtimes_r \mathbb{F}_\infty)=0$.
We also have $K_1(C(X)\rtimes_r \mathbb{F}_\infty) \cong \mathbb{Z}^{\oplus \infty}$ for any Cantor system of $\mathbb{F}_\infty$.
This with the classification theorem of Kirchberg--Phillips \cite{Kir}, \cite{Phi} shows that
generically the crossed products give only a single \Cs -algebra.
However, we remark that there are continuously many Kirchberg algebras which are realized as
the reduced crossed product of an amenable minimal free Cantor system of $\mathbb{F}_\infty$ \cite{Suz}.
\end{Rem}

The next proposition says that property $\mathcal{R}$
implies the non-existence of nontrivial $\Gamma$-invariant closed subspace of $C(X)$.
This result may be of independent interest.
\begin{Prop}\label{Prop:inv}
Assume $\alpha\in \mathcal{S}(\Gamma, X)$ satisfies $\mathcal{R}$.
Then there is no $\Gamma$-invariant closed subspace of $C(X)$
other than $0$, $\mathbb{C}$, or $C(X)$.
In particular $\mathcal{R}$ implies primeness.
\end{Prop}
\begin{proof}
Let $V$ be a closed $\Gamma$-invariant subspace of $C(X)$
other than $0$ or $\mathbb{C}.$
We first show that $V$ contains $\mathbb{C}$.
Take a nonzero function $f\in V$. 
Then for any $\epsilon>0$,
there is a partition $\mathcal{U}:=\{ U_1, \ldots, U_n\}$ of $X$ by proper clopen sets
and complex numbers $c_1, \ldots, c_n$ with $|c_1|=\|f\|$
such that with $g:=\sum_{i=1}^n c_i \chi_{U_i}$,
we have $\|f-g\|<\epsilon$.
Put $c:=\frac{1}{n} \sum_{i=1} ^n c_i$.
By replacing $\mathcal{U}$ by dividing $U_1$ into sufficiently many clopen subsets and replacing the sequence $(c_i)_i$ suitably,
we may assume $|c| \geq \|f\|/2$.
By property $\mathcal{R}$, we can take $s\in \Gamma$ with
$sU_i= U_{i+1}$ for all $i$.
We then have
$\sum_{i=1}^n s^i gs^{-i}=\sum_{i=1}^n c_i$.
This yields the inequality
\[\|\frac{1}{n}\sum_{i=1}^n s^ifs^{-i} - c\|< \epsilon.\]
Since $\epsilon>0$ is arbitrary and $|c| \geq \|f\|/2$,
we obtain $\mathbb{C}\subset V$.

From this, we can choose a nonzero function $f\in V$
with $0\in f(X)$.
For any $\epsilon >0$,
take a partition $\mathcal{U}=\{U_0, U_1, \ldots, U_n\}$ of $X$ by proper clopen sets
and complex numbers $c_1, \ldots, c_n$
such that with $g:=\sum_{i=1}^n c_i \chi_{U_i}$,
we have $\|f-g\|<\epsilon$.
Put $c:=\frac{1}{n}\sum_{i=1}^n c_i$.
As before, we may assume $|c|\geq \|f\|/2$.
By using property $\mathcal{R}$ to the family $\{U_1, \ldots, U_n\}$, we can take $s\in \Gamma$ satisfying
$sU_0=U_0$ and $sU_i=U_{i+1}$ for $1\leq i <n$.
Then we have $\frac{1}{n}\sum_{i=1}^{n} s^ig s^{-i} =c\chi_{X\setminus U_0}$.
Now let $U$ be any proper clopen subset of $X$.
Take $t\in \Gamma$ with $t(X \setminus U_0)=U$.
(To find such $t$, use property $\mathcal{R}$ twice.)
We then have 
\[t(\frac{1}{n}\sum_{i=1}^{n} s^ig s^{-i})t^{-1} =ct(\chi_{X\setminus U_0})t^{-1}=c\chi_U.\]
This shows the inequality
\[\|(\frac{1}{n}\sum_{i=1}^{n} ts^if s^{-i}t^{-1}) - c\chi_U\|< \epsilon.\]
Since $\epsilon>0$ is arbitrary, this proves $\chi_U\in V$.
Since $U$ is arbitrary, we obtain $V=C(X)$.
\end{proof}

We need the following restricted version of the Powers property for free product groups.
Although the proof is essentially contained in \cite{Pow}, for completeness, we include a proof.
\begin{Lem}[Compare with Lemma 5 of \cite{Pow} and Lemma 5 of \cite{dlS}]\label{Lem:Pow}
Let $\Lambda_1, \Lambda_2$ be groups
and set $\Lambda:=\Lambda_1 \ast \Lambda_2$.
Let $s\in \Lambda_1$, $t\in \Lambda_2$ be
torsion-free elements.
Then for any finite subset $F$ of $\Lambda\setminus \{e\}$,
there are a partition $\Lambda=D\sqcup E$ of $\Lambda$
and elements $u_1, u_2, u_3\in \langle s, t \rangle$ with the following properties. 
\begin{enumerate}[\upshape(1)]
\item
$fD\cap D= \emptyset$ for all $f\in F$.
\item
$u_jE\cap u_kE=\emptyset$ for any two distinct $j, k \in \{1, 2, 3\}$.
\end{enumerate}
\end{Lem}
\begin{proof}
Let $F \subset \Lambda\setminus \{e\}$ be given.
Then for sufficiently large $n\in \mathbb{N}$, with $z:=ts^n$,
any element of
$zFz^{-1}$
is started with $t$ and ended with $t^{-1}$.
Here for $u \in \Lambda_i \setminus \{e \}$,
we say an element $w$ of $\Lambda$ is started with $u$
if $w=uw_1\ldots w_n$ for some (possibly empty) sequence $w_1, \ldots, w_n$ with $w_j\in \Lambda_{k_j}\setminus \{e\}$ and
$i\neq k_1 \neq k_2 \neq \cdots \neq k_n$.
The word ``ended with $u$'' is similarly defined.
(Thus, in our terminology, the element $u^2$ is not started with $u$.)

Let $E'$ be the subset of $\Lambda$
consisting of all elements started with $t$.
Put $E:=z^{-1}E'$,
$D:=\Lambda \setminus E$, and $D':=\Lambda \setminus E'$.
Then note that $fD\cap D=\emptyset$ for all $f\in F$ if and only if
$f'D'\cap D'=\emptyset$ for all $f' \in zFz^{-1}$.
Since elements $f' \in zFz^{-1}$ are started with $t$ and ended with $t^{-1}$
but $D'$ consists of elements not started with $t$,
we have $f'D'\cap D'=\emptyset.$
Now for $j\in \{1, 2, 3\}$, put $u_j:=s^j z$.
Obviously each $u_j$ is contained in $\langle s, t\rangle$.
By definition, we have $u_jE=s^jE'$.
This shows that $u_jE$ consists of only elements started with $s^j$.
Therefore $u_1 E, u_2 E$, and $u_3 E$ are pairwise disjoint.
\end{proof}

Now we prove Main Theorem.
Before the proof, we remark that the AP is preserved under taking free products.
Hence $\Gamma$ has the AP if and only if each free product component $\Gamma_i$ has it.
See Section 12.4 of \cite{BO} for the detail.
\begin{Thm} \label{Thm:main}
Let $\Gamma$ be an infinite free product group with the AP.
Then, for $\alpha\in \mathcal{S}(\Gamma, X)$ with property $\mathcal{R}$,
there is no proper intermediate \Cs -algebra of the inclusion
${\rm C}^\ast_r(\Gamma)\subset C(X)\rtimes_r \Gamma$.
In particular, when additionally $\alpha$ is amenable,
then $C(X)\rtimes_r \Gamma$ is a minimal ambient nuclear \Cs -algebra
of the non-nuclear \Cs -algebra ${\rm C}^\ast_r(\Gamma)$.
\end{Thm}
\begin{proof}
Let $A$ be an intermediate \Cs -algebra of the inclusion
${\rm C}^\ast_r(\Gamma)\subset C(X)\rtimes_r \Gamma$.
We first consider the case $E(A)=\mathbb{C}$.
In this case, thanks to Theorem 3.2 of \cite{Zac} (see also Proposition 3.4 of \cite{Suz2}),
we have the equality $A={\rm C}^\ast_r(\Gamma)$.

We next consider the case $E(A)\neq \mathbb{C}$.
In this case, by Proposition \ref{Prop:inv}, $E(A)$ is dense in $C(X)$.
Let $U$ be a proper clopen subset of $X$.
Let $\epsilon>0$ be given.
Then take a self-adjoint element $x\in A$ with $\| E(x)-\chi_U\|<\epsilon$.
By property $\mathcal{R}$, there are torsion-free elements $s_1\in \Gamma_i$ and $s_2 \in \Gamma_j$ with $i\neq j$ 
which fix $\chi_U$. Put $\Lambda:=\langle s_1, s_2 \rangle$.
Take $y\in C(X)\rtimes_{\rm alg} \Gamma$ satisfying
$E(y)=\chi_U$ and $\|y-x\|<\epsilon$.
By Lemma \ref{Lem:Pow},
we can apply the Powers argument, Lemma 5 of \cite{dlS}, by elements of $\Lambda$.
Iterating the Powers argument sufficiently many times, we obtain a sequence $t_1, \ldots, t_n \in \Lambda$
satisfying the inequality
\[\|\frac{1}{n}\sum_{i=1}^n t_i(y-\chi_U)t_i^{-1}\|<\epsilon.\]
Since $\chi_U$ is $\Lambda$-invariant,
we have
\[\|\frac{1}{n}\sum_{i=1}^n t_i x t_i^{-1}-\chi_U \|<2\epsilon.\]
Since $\epsilon >0 $ is arbitrary,
this shows $\chi_U\in A$.
Therefore $A=C(X)\rtimes_r\Gamma$.
\end{proof}
\begin{Rem}\label{Rem:ma}
It is impossible to find a minimal ambient nuclear \Cs -algebra
of a non-nuclear \Cs -algebra by maximality arguments.
From the outside, it is shown in \cite{Suz2} that
the decreasing sequence of nuclear \Cs -algebras need not be nuclear.
From the inside, it can be shown that the increasing union of 
non-nuclear \Cs -algebras can be nuclear.
Here we give an example.
Let $A$ be a unital nuclear \Cs-algebra and let $B$ be a non-nuclear \Cs -subalgebra of $A$ containing the unit of $A$.
Put $A_n:=A^{\otimes n}\otimes B$ for each $n\in \mathbb{N}$.
Then they are canonically identified with \Cs -subalgebras
of the infinite tensor power $A^{\otimes \infty}$ of $A$.
Each $A_n$ is not nuclear
but their increasing union is the
nuclear \Cs -algebra $A^{\otimes \infty}$.
\end{Rem}
\begin{Rem}
When $\Gamma$ is exact without the AP (e.g., $\Gamma={\rm SL}(3, \mathbb{Z})$ \cite{LS}),
for any amenable Cantor system
$\Gamma \curvearrowright X$ with $\mathcal{R}$,
proper intermediate \Cs -algebras of ${\rm C}^\ast_r(\Gamma)\subset C(X)\rtimes_r \Gamma$.
are contained in the following \Cs -algebra.
\[A:=\{x\in C(X)\rtimes_r\Gamma: E(xs)\in \mathbb{C}{\rm \ for\ all\ }s\in \Gamma\}.\]
By Proposition 2.4 in \cite{Suz2}, any such a \Cs -algebra does not have the OAP.
Hence $C(X)\rtimes_r \Gamma$ is a minimal ambient nuclear \Cs -algebra of ${\rm C}^\ast_r(\Gamma)$
and is an ambient nuclear \Cs -algebra of the non-OAP \Cs -algebra $A$
with no proper intermediate \Cs -algebras.
\end{Rem}
\begin{Rem}
Let $\Gamma$ be an exact infinite free product group.
We show that amenable Cantor systems of $\Gamma$ with property $\mathcal{R}$ are not unique at the level of continuous orbit equivalence. (Recall that two free Cantor systems are continuously orbit equivalent
if and only if their transformation groupoids are isomorphic. Here we regard it as a definition. For the precise definition,
see Definition 5.4 of \cite{Suz} for instance.)
We also show that minimal ambient nuclear \Cs -algebras of ${\rm C}^\ast_r(\Gamma)$
are not unique in the following sense:
there is no isomorphism between them that is identity on ${\rm C}^\ast_r(\Gamma)$.
We say that two ambient \Cs -algebras are conjugate if such an isomorphism exists.
Let $[[\alpha]]$ denote the topological full group of $\alpha$.
That is, the group of homeomorphisms $\varphi$ on $X$ which admit a partition $(U_s)_{s\in \Gamma}$ of $X$
by clopen subsets with
$\varphi|_{U_s}=\alpha_s|_{U_s}$ for all $s\in \Gamma$.
Then the set
\[E[[\alpha]]:=\{\theta \in \mathbb{R}/\mathbb{Z}: {\rm some\ } \varphi \in [[\alpha]] {\rm \ factors\ the \ rotation\ }
R_\theta\colon \mathbb{T} \rightarrow \mathbb{T} \}\]
is invariant under continuous orbit equivalence.
Note that $E[[\alpha]]$ is countable by the metrizability of $X$.
For any amenable Cantor system $\beta \colon \Gamma \curvearrowright Y$,
by working on the closed subspace
$\{\alpha \in \mathcal{S}(\Gamma, X):\alpha|_{\Gamma_1}=(\beta^{\otimes \mathbb{N}})|_{\Gamma_1}\}$
instead of $\mathcal{S}(\Gamma, X)$,
we can find an amenable Cantor system $\alpha$ with $\mathcal{R}$ in this set.
Here we identify $X$ with $Y^\mathbb{N}$
and we denote by $\beta^{\otimes \mathbb{N}}$ the diagonal action of infinitely many copies of $\beta$.
Hence, with the aid of (a modification of) Lemme \ref{Lem:ext}, for any irrational number $\theta$, we can find an amenable free Cantor system $\alpha$ of $\Gamma$ with $\mathcal{R}$
satisfying $\theta \in E[[\alpha]]$.
Thus there is a family of continuously many amenable free Cantor systems with $\mathcal{R}$
whose members are pairwise not continuously orbit equivalent.
We show that their crossed products give pairwise non-conjugacy ambient \Cs -algebras.
Suppose two of them are conjugate.
Then the composite of a conjugating isomorphism
with the canonical conditional expectation gives a $\Gamma$-equivariant unital completely positive map
between two $C(X)$. This is impossible by Lemma 3.10 of \cite{KK} (with Remark \ref{Rem:P}) and Proposition \ref{Prop:inv}.
\end{Rem}

\section{Further examples}\label{sec:ex}
We close this paper with the following result on minimal tensor products.
This emphasizes the tightness of ambient \Cs -algebras obtained in Theorem \ref{Thm:main}.
\begin{Prop}\label{Prop:tensor}
Let $A$ be a simple \Cs -algebra.
Let $\Gamma$ be an infinite free product group with the AP.
Let $\alpha \colon \Gamma \curvearrowright X$ be a Cantor system with property $\mathcal{R}$.
Then the inclusion $A\otimes {\rm C}^\ast_r(\Gamma) \subset A\otimes (C(X)\rtimes _r \Gamma)$ has no proper intermediate \Cs -algebra.
\end{Prop}
\begin{proof}
Let $B$ be an intermediate \Cs -algebra of the inclusion
$A\otimes {\rm C}^\ast_r(\Gamma) \subset A\otimes (C(X)\rtimes _r \Gamma)$.
Put $\Phi:=\id _A \otimes E$.
Throughout the proof, we identify $A$ with a \Cs -subalgebra of $A\otimes C(X)$ in the canonical way.
Note that the image $\Phi(B)$ contains $A$.
When the equality
$\Phi(B)=A$ holds,
by Proposition 3.4 of \cite{Suz2} (with Exercise 4.1.3 of \cite{BO}),
we have $B=A\otimes {\rm C}^\ast_r(\Gamma)$.

Suppose $\Phi(B)\neq A$.
We observe first that for an element $x\in A\otimes C(X)$ satisfying
$(\varphi \otimes \id_{C(X)})(x) \in \mathbb{C}$ for all pure states $\varphi$ on $A$,
we have $x=(\id_A\otimes \psi)(x)\in A$ for any state $\psi$ on $C(X)$.
Hence we can choose a pure state $\varphi$ on $A$ and an element $b\in B$ satisfying
$f:=(\varphi\otimes \id_{C(X)})(\Phi(b)) \in C(X) \setminus \mathbb{C}$ and $\| f\| =1$.
Now let $\epsilon >0$ be given.
By the Akemann--Anderson--Pedersen excision theorem (Theorem 1.4.10 of \cite{BO}),
there is a positive element $a\in A$ of norm one satisfying
$\|\Phi(a ba)- a^2\otimes f \|< \epsilon$.
By the simplicity of $A$, for any $\epsilon>0$ and any positive contractive element $c \in A$,
there is a finite sequence $x_1, \ldots, x_n \in A$
satisfying the following conditions.
\begin{itemize}
\item
$\| \sum_{i=1}^n x_i a^2 x_i^\ast -c \|\leq \epsilon$.
\item
$\|\sum_{i=1}^n x_i x_i^\ast\| \leq 2$.
\end{itemize}
Indeed, when $A$ is unital, this follows from Lemma 2.3 of \cite{Zac0}.
When $A$ is non-unital, this follows from Proposition 1.4.5 of \cite{Ped}
and the proof of Lemma 2.3 of \cite{Zac0}.

For such a sequence,
we have
\[\| \Phi(\sum_{i=1}^n x_i a b a x_i^\ast)-c\otimes f)\|<3\epsilon.\]
This shows that the closure of $\Phi(B)$ contains $c \otimes f$.
Proposition \ref{Prop:inv} then shows that
the closure of $\Phi(B)$ contains the subspace $c \otimes C(X)$.
Now the proof of Theorem \ref{Thm:main} shows the equality $B = A\otimes (C(X)\rtimes _r \Gamma)$.
\end{proof}
\begin{Rem}
When $A$ is nuclear (or more generally, when $A$ has Wassermann's property (S) \cite{Was}),
Proposition \ref{Prop:tensor} follows from Theorem \ref{Thm:main} and tensor splitting results \cite{Zac0}, \cite{Zsi}.
\end{Rem}
\subsection*{Acknowledgement}
The author would like to thank David Kerr
for his suggestion to study generic properties of dynamical systems of non-amenable groups.
He thanks Narutaka Ozawa for letting him know the papers \cite{Zac0} and \cite{Zsi}.
He also thanks Yasuyuki Kawahigashi, who is his advisor, for helpful comments on the first draft of the paper.
This work was carried out while the author was visiting the University of M\"{u}nster
under the support of Program of Leading Graduate Schools, MEXT, Japan.
He acknowledges their hospitality.
The author was supported by Research Fellow
of the JSPS (No.25-7810) and Program of Leading Graduate Schools, MEXT, Japan.

\end{document}